\documentclass{amsart}
\usepackage
  { tikz
  , scalerel
  , amssymb
  , amsthm
  , amsrefs
  , enumitem
  , hyperref
  , blkarray
  , aliascnt
  }
\usepackage[dvipsnames]{xcolor}
\usetikzlibrary{automata,positioning,arrows,arrows.meta,shapes.geometric}
\definecolor{ourgreen}{cmyk}{1,0,1,0}

\definecolor{ourblue}{HTML}{2867FF}
\definecolor{ourred}{HTML}{FF2828}
\newcommand{\blue}{\color{ourblue}}
\newcommand{\red}{\color{ourred}}
\newcommand{\green}{\color{ourgreen}}
\NewDocumentCommand{\magenta}{}{\color{magenta}}

\NewDocumentCommand{\bb}{}{{\blue \raisebox{.4pt}{\scaleobj{.3}{\begin{tikzpicture}
    \node[isosceles triangle,
    isosceles triangle apex angle=60,rotate=90,draw,very thick,fill] (p) at (0,0) {};
    \draw[draw=none] (0,0.15) circle (.275) {};
\end{tikzpicture}}}}} 
\NewDocumentCommand{\Bb}{}{{\blue \raisebox{.4pt}{\scaleobj{.3}{\begin{tikzpicture}
    \node[isosceles triangle,
    isosceles triangle apex angle=60,rotate=90,draw,line width=2pt] (p) at (0,0) {};
    \draw[draw=none] (0,0.15) circle (.275) {};
\end{tikzpicture}}}}}
\NewDocumentCommand{\ar}{}{{\red {\scaleobj{.4}{\begin{tikzpicture}
    \node[circle,draw,thick,fill=red] (p) at (0,0) {};
    \draw[draw=none] (0,0) circle (.2) {};
\end{tikzpicture}}}}}
\NewDocumentCommand{\Ar}{}{{\red {\scaleobj{.4}{\begin{tikzpicture}
    \node[circle,draw,ultra thick] (p) at (0,0) {};
    \draw[draw=none] (0,0) circle (.2) {};
\end{tikzpicture}}}}}
\NewDocumentCommand{\cg}{}{{\green 
\raisebox{.4pt}{\scaleobj{.5}{\begin{tikzpicture}
    \node[regular polygon sides=4,draw,thick,fill] (p) at (0,0) {};
    \draw[draw=none] (0,.07) circle (.175) {};
\end{tikzpicture}}}}}
\NewDocumentCommand{\dm}{}{\raisebox{.3pt}{\scaleobj{.4}{\magenta \begin{tikzpicture}
    \node[regular polygon,draw,thick,fill=magenta] (p) at (0,0) {};
    \draw[draw=none] (0,0.07) circle (.225) {};
\end{tikzpicture}}}}

\NewDocumentCommand{\Dm}{}{\raisebox{.3pt}{\scaleobj{.4}{\magenta \begin{tikzpicture}
    \node[regular polygon,draw,ultra thick] (p) at (0,0) {};
    \draw[draw=none] (0,0.07) circle (.225) {};
\end{tikzpicture}}}}

\NewDocumentCommand{\Cg}{}{{\green 
\raisebox{.4pt}{\scaleobj{.5}{\begin{tikzpicture}
    \node[regular polygon sides=4,draw,very thick] (p) at (0,0) {};
    \draw[draw=none] (0,.07) circle (.175) {};
\end{tikzpicture}}}}}
\NewDocumentCommand{\bigbb}{}{{\blue \raisebox{0.5pt}{\scaleobj{.3}{\begin{tikzpicture}
    \node[isosceles triangle,
    isosceles triangle apex angle=60,rotate=90,draw,ultra thick,fill] (p) at (0,0) {};
\end{tikzpicture}}}}}
\NewDocumentCommand{\bigar}{}{{\red \raisebox{-1.5pt}{\scaleobj{.4}{\begin{tikzpicture}
    \node[circle,draw,very thick,fill] (p) at (0,0) {};
\end{tikzpicture}}}}}
\NewDocumentCommand{\bigcg}{}{{\green 
\raisebox{0.2pt}{\scaleobj{.45}{\begin{tikzpicture}
    \node[regular polygon, regular polygon sides=4, draw, fill] (p) at (0,0) {};
\end{tikzpicture}}}}}
\NewDocumentCommand{\bigdm}{}{\scaleobj{.45}{\magenta \begin{tikzpicture}
    \node[regular polygon,draw,fill=magenta] (p) at (0,0) {};
\end{tikzpicture}}}

\NewDocumentCommand{\bigBb}{}{{\blue \raisebox{0.5pt}{\scaleobj{.3}{\begin{tikzpicture}
    \node[isosceles triangle,
    isosceles triangle apex angle=60,rotate=90,draw,line width=3pt] (p) at (0,0) {};
\end{tikzpicture}}}}}
\NewDocumentCommand{\bigAr}{}{{\red \raisebox{-1.5pt}{\scaleobj{.4}{\begin{tikzpicture}
    \node[circle,draw,line width=2.25pt] (p) at (0,0) {};
\end{tikzpicture}}}}}
\NewDocumentCommand{\bigCg}{}{{\green 
\raisebox{0.2pt}{\scaleobj{.45}{\begin{tikzpicture}
    \node[regular polygon, regular polygon sides=4, draw, line width=2pt] (p) at (0,0) {};
\end{tikzpicture}}}}}

\NewDocumentCommand{\growth}{}{\mathfrak{g}}
\NewDocumentCommand{\PP}{m}{\operatorname{PP}_#1}
\NewDocumentCommand{\goldC}{}{\hyperref[gold]{Goldstein's Criterion}}
\NewDocumentCommand{\defnsb}{m}{\textbf{#1}}

\tikzset{>={Classical TikZ Rightarrow[length=1.375mm]}}

\counterwithin{figure}{section}

\newaliascnt{theorem}{figure}
\newtheorem{theorem}[theorem]{Theorem}

\aliascntresetthe{theorem}
\newaliascnt{lemma}{figure}
\newtheorem{lemma}[lemma]{Lemma}

\aliascntresetthe{lemma}
\newaliascnt{proposition}{figure}
\newtheorem{proposition}[proposition]{Proposition}

\aliascntresetthe{proposition}
\newaliascnt{corollary}{figure}
\newtheorem{corollary}[corollary]{Corollary}

\aliascntresetthe{corollary}
\newaliascnt{question}{figure}
\newtheorem{question}[question]{Question}

\aliascntresetthe{question}

\theoremstyle{definition}
\newaliascnt{algorithm}{figure}

\aliascntresetthe{algorithm}
\newaliascnt{definition}{figure}
\newtheorem{definition}[definition]{Definition}

\aliascntresetthe{definition}
\newaliascnt{example}{figure}

\aliascntresetthe{example}

\theoremstyle{remark}
\newaliascnt{remark}{figure}

\aliascntresetthe{remark}

\newlist{steps}{enumerate}{3}
\setlist[steps]{label=Step \arabic*.,ref=Step \arabic*,leftmargin=15mm}

\newlist{questions}{enumerate}{3}
\setlist[questions]{label=Q.\ \arabic*.,ref=Question \arabic*,leftmargin=15mm}

\title[Growth and language complexity of potential positivity]{Growth and language complexity of potentially positive elements in free groups}

\author[]{Emma Dinowitz}
\address{Department of Mathematics, CUNY Graduate Center, New York,
NY 10016}  \email{emmad4867@gmail.com}

\author[]{Lucy Koch-Hyde}
\address{Department of Mathematics, CUNY Graduate Center, New York,
NY 10016}  \email{lhyde@gradcenter.cuny.edu}

\author[]{Siobh\'an O'Connor}
\address{Department of Mathematics, CUNY Graduate Center, New York,
NY 10016}  \email{doconnor@gradcenter.cuny.edu}

\author[]{
  \'Eamonn Olive
}
\email{ejolive97@gmail.com}

\begin{document}%
\begin{abstract}
An word in a free group is called ``potentially positive'' if it is automorphic to an element which is written with only positive exponents.
We will develop automata to analyze properties of potentially positive words.
We will use these to give new bounds on the asymptotic growth of potentially positive elements in free groups of 2 to 7 generators.
We prove the bounds for $F_2$ are tight, giving the growth function up to a constant multiplier.
We use the same tools to show that certain restricted automata cannot recognize the set of potentially positive elements.
\end{abstract}%
\maketitle%
\section{Introduction}
A word in a free group $F_r=F(x_1,\dots,x_r)$ is called \defnsb{positive} if it can be expressed as a product of the generators, $x_1,\dots, x_r$, without inverses.

Groups which can be given as a presentation with a single positive relator are called \textbf{positive one-relator groups}.
Several results are known about this class of groups:
\begin{itemize}
\item It is shown in \cite{Baumslag} that positive one-relator groups are residually solvable. \cite{Linton} strengthens this by proving that positive one-relator groups are free-by-solvable.\footnote{\cite{Linton} also shows that torsion-free positive one-relator groups satisfy an even stronger property. They are free-by-($\mathbb Q$-solvable).}
\item It is shown in \cite{Wise} that positive one-relator groups satisfying the $C'(1/6)$ small cancellation condition are residually finite.
\end{itemize}
Unfortunately, the probability that a randomly selected word of a given length is positive vanishes exponentially as the length increases; the number of positive words of length $n$ is $r^n$ while the total number of words of length $n$ is $2r\left(2r-1\right)^{n-1}$.

However, applying a free group automorphism to the relators of a group presentation produces an isomorphic group.
Thus if a word is automorphic to a positive word, its associated one-relator group has properties analogous to the ones above.
Such words are called \textbf{potentially positive}, and were first introduced in \cite{OP2002}.

This motivates our main topic of investigation:
\begin{question}\label{Q:growth}
How many potentially positive words are there of length $n$ in $F_r$? i.e.\ How does the set of potentially positive words grow?
\end{question}

Previous work has calculated asymptotic approximations of the growth of other similar classes of words. 
We call words that are automorphic to a generator \defnsb{primitive}.
\cite{F2prim} gives the growth of primitive words in $F_2$ as $\sqrt{3}^n$, and \cite{PuderWu} gives the growth of primitive words in $F_r$ as $(2r-3)^n$ when $r>2$.
(These are asymptotic approximations. We will formally define what sort of approximations these are later.)
Since primitive words are also potentially positive, these form a starting lower bound for the growth of potentially positive words.
This also tells us that, apart from rank 2, the potentially positive elements grow significantly faster than the positive elements.
In fact the density of the positive elements among the potentially positive elements goes to zero as the size increases.

The main result of this paper answers \autoref{Q:growth} for $F_2$.
More specifically, we prove the following theorem.
\begin{theorem}\label{thm:main}
  The number of potentially positive words in $F_2$ of length $n$ is $\Theta(\lambda_1^n)$, where $\lambda_1 \approx 2.50506841362147$ is the largest root of the polynomial
  \begin{align}
    \lambda^4 - 3 \lambda^3 + \lambda^2 + \lambda - 1
  \end{align}
\end{theorem}
From this calculation we see that potentially positive words do exponentially vanish in the set of all reduced words in $F_2$, but at a slower rate than positive words.

We also give novel lower bounds for the growth of potentially positive words in $F_r$ for $3\leq r\leq 7$.
To our knowledge, the best lower bounds for ranks greater than 7 remain those coming from \cite{PuderWu}. Further progress on \autoref{Q:growth} likely requires an algorithm that recognizes potentially positive words in $F_r$; our answer in $F_2$ relies on one for $F_2$ first provided by Goldstein in \cite{Goldstein}, and improved in \cite{KHOO}. It is not known whether such an algorithm exists for $r>2$, thus we ask:

\begin{question}
How complicated is it to decide whether a word is potentially positive?
\end{question}

In particular, we would like to determine what sorts of automata can recognize potential positivity.
The existing algorithms show that for rank two anything at least as powerful as a Turing machine can decide it.
In fact \cite{KHOO} provides an algorithm which uses worst-case linear memory and so the rank two case can be decided with a linear bounded automaton.

In this paper we will approach the question in the opposite direction.
We will show that certain models of computation (including finite state machines and push down automata) are insufficient for deciding $F_r$ for $r\geq 2$.

\section{Background}
\subsection{Growth and density}

This paper deals heavily with the growth of sets of words.

\begin{definition}
Let $S$ be a set of words on a finite alphabet $\Sigma$.
We define the \defnsb{growth function} $\mathfrak g_n(S)$ to be the number of the words in $S$ with length $n$.
\end{definition}

We will define several asymptotic ways to consider growth for these functions.
The finest notion of growth we will use is asymptotic equivalence.
\begin{definition}
We say two functions $f, g : \mathbb N\rightarrow \mathbb N$ are \defnsb{asymptotically equivalent}, written $f \sim g$, when
\[
\lim_{n\rightarrow\infty} \dfrac{f(n)}{g(n)} = 1
\]
\end{definition}

Coarser notions of growth will be useful.

\begin{definition}\label{dfn:comp}
We will have the following classes of $\mathbb N \rightarrow \mathbb N$ functions:
\begin{itemize}
\item $O(g)= \left\{ f\mid\exists C,m:\forall n\geq m: Cg(n)> f(n)\right\}$
\item $\Omega(g)= \left\{ f\mid g\in O(f)\right\}$
\item $\Theta(g)= \Omega(g)\cap O(g)$
\end{itemize}
\end{definition}

These create a partial order, being analogous to $\leq$, $\geq$, and $=$, respectively.
The most well known of these notions is certainly the first, called ``big-O'', which is commonly used in algorithmic analysis.
We will make use of each of them.
We will also make sparing use of the classes analogous to $<$ and $>$:

\begin{itemize}
\item $o(g)= O(g)\setminus\Theta(g)$
\item $\omega(g)= \Omega(g)\setminus \Theta(g)$
\end{itemize}

\begin{definition}
A set $B\subseteq A\subseteq \Sigma^\ast$ has a \defnsb{density function} defined as
\[
n \mapsto\dfrac{\mathfrak g_n(B)+1}{\mathfrak g_n(A)+1}
\]
\end{definition}

\begin{definition}
A set $B\subseteq A$ is \defnsb{negligible} in $A$ when the limit of the density is $0$.
Alternatively $B$ is \defnsb{generic} in $A$ when the limit of the density is $1$.
\end{definition}

A stronger condition will also be useful.

\begin{definition}
A set $B\subseteq A$ with density function $d$ is \defnsb{exponentially negligible} in $A$ when $d(n) \in O\left(C^{-n}\right)$ for some constant $C > 1$.
Likewise $B$ is \defnsb{exponentially generic} when its relative complement, $A\setminus B$, is exponentially negligible.
\end{definition}


\subsection{Free groups}
In order to increase the visual contrast between letters we will generally write the generators of a free group using colored shapes ($\ar$, $\bb$, $\cg$, and $\dm$).
We will represent the inverse of a generator with the outline of the corresponding shape ($\Ar$, $\Bb$, $\Cg$, and $\Dm$ respectively).

We will represent automorphisms by showing the action only on the affected generators.
For example, in $F(\ar,\bb)$ we use $\bb \mapsto \bb \ar$ to indicate the automorphism which performs that mapping and behaves as the identity on $\ar$.

We say a word is \defnsb{cyclically reduced} if its first and last letters are not each other's inverses. 
Similarly, we can find a word's \defnsb{cyclic reduction} by deleting its first and last letter until it is cyclically reduced.
Since cyclic reduction is always a conjugation, the cyclic reduction of a word is always in the same automorphic orbit as the word.

We will see in the following results that finding the growth rate of the cyclically reduced words is often sufficient for determining the growth rate of all words.
In fact, if a set of words grows faster than the conjugacy class of one of the words, then the growth of the cyclically reduced words is the same as that of the whole set.
We will see this in \autoref{cor:CRisenough}, but first we need the following technical lemma.

\begin{lemma}\label{lem:maxgro}
  Let us have two functions $f$ and $g$, and some $b > a > 1$ such that $f(n) \in O(a^n)$ and $g(n) \in \Omega(b^n)$.
  Roughly speaking, the exponential part of $g$ has a greater base than the exponetial part of $f$.
  The function $x\mapsto \sum_{i=0}^x f(i)g(x-i)$ is in $\Theta(g)$.
\end{lemma}
\begin{proof}
    The lower bound is trivial, so we will proceed to the upper bound.

    Observe the following.
    \begin{align*}
    O\left(\sum_{i=0}^x f(i)g(x-i)\right) &\subseteq O\left(\sum_{i=1}^x \dfrac{f(i)g(x)}{b^i}\right) \tag{1} \label{eqn:recsub}\\
    &\subseteq O\left(g(x)\sum_{i=1}^x\dfrac{a^i}{b^i}\right) \tag{2} \label{eqn:asub} \\
    &\subseteq O\left(g(x)\sum_{i=1}^\infty\dfrac{a^i}{b^i}\right) \\
    &= O\left(g\right) \tag{3} \label{eqn:conv} \\ 
    \end{align*}
    (\ref{eqn:recsub}) holds because $g(n)\in \Omega(b^n)$ implies $bg(x-1) \leq g(x)$ and thus $g(x-i)\leq b^{-1}g(x)$.
    (\ref{eqn:asub}) holds because $f(n) \in O(a^n)$.
    (\ref{eqn:conv}) holds because $b > a$ and thus the series converges to $(1-b/a)^{-1}$, a constant.
\end{proof}

\begin{corollary}\label{cor:CRisenough}
    Let $S$ be a subset of $F_r$ that contains only cyclically reduced words.
    Let $\bar S$ be the set of words whose cyclic reductions are in $S$.
    If the growth of words of length $\ell$ in $S$ is $g \in \Omega\left(b^\ell\right)$ for some $b> \sqrt{2r-1}$, then the growth of words of length $\ell$ in $\bar S$ is $\Theta(g)$.
\end{corollary}
\begin{proof}
    Let us consider words in $\bar S$ of length $\ell$ and cyclically reduced length $\ell'$.
    These words are of the form $vwv^{-1}$ where $w$ is in $S$ and there is no cancellation between $v$ and $w$ or between $w$ and $v^{-1}$.
    It follows that $|w|=\ell'$ and $|v|=(\ell-\ell')/2$.
    From our hypothesis, the growth of words of length $\ell'$ in $S$ is some function $g$.
    $v$ can be any word of length $(\ell-\ell')/2$ so long as its last letter does not cancel with the first letter of $w$ or that of $w^{-1}$.
    The growth of such words is some function $f(\ell-\ell') \in \Theta\left((2r-1)^{(\ell-\ell')/2}\right)=\Theta\left(\sqrt{2r-1}^{\ell-\ell'}\right)$.
    Thus the growth of words of this form is $g(\ell')f(\ell-\ell')$. 
    The set of all words in $\bar S$ of length $\ell$ is then the sum for every $\ell'$ between 0 and $\ell$.
    Thus by \autoref{lem:maxgro} this is $\Theta(g)$.
\end{proof}

This corollary will be applicable to all of the sets addressed in the paper.
We are lucky enough that our growth is always significantly higher than that of a conjugacy class.
However, an analogous argument can be applied to show the following as well.

\begin{corollary}
    Let $S$ be a subset of $F_r$ that contains only cyclically reduced words.
    Let $\bar S$ be the set of words whose cyclic reductions are in $S$.
    If the growth of words of length $\ell$ in $S$ is $g \in O\left(b^\ell\right)$ for some $b< \sqrt{2r-1}$, then the growth of words of length $\ell$ in $\bar S$ is $\Theta\left(\sqrt{2r-1}^\ell\right)$.
\end{corollary}

Cyclically reduced primitive words in $F_2$ grow as a polynomial, and so the total set grows at a rate of $\sqrt{3}$, the same rate as a conjugacy class.

\subsection{Potential positivity}

We will need a few facts about potential positivity, especially in $F_2$. First, we present a basic but very useful fact about potential positivity in free groups of any rank.

\begin{proposition}\label{prop:allbutone}
    Suppose a word $w$ in $F_r=F(x_1,\dots,x_r)$, $r\geq 2$, has $r-1$ generators that only occur with positive exponents. Then $w$ is potentially positive.
\end{proposition}
\begin{proof}
    Suppose without loss of generality that every generator except $x_r$ only occurs with positive exponents. Then every syllable of $x_r^m$ where $m$ is negative occurs after a syllable of another generator to a positive power. Let $k$ be the largest negative exponent in absolute value that appears on $x_r$ in $w$. Since $x_1,\dots,x_{r-1}$ only appear with positive exponents, applying the automorphism $x_i\mapsto x_ix_r^{-k}$ for each $i\in\{1,2,\dots, r-1\}$ sends $w$ to a positive word.
\end{proof}

We will make extensive use of the following theorem of Goldstein \cite{Goldstein} (for an alternate proof, see \cite{KHOO}).
\begin{theorem}[Goldstein's Criterion]\label{gold}
Let $w$ be cyclically reduced and potentially positive.
If we consider $w$ as a cyclic word,  there exists $x,y\in\left\{\Ar,\ar,\Bb,\bb \right\}$ such that every occurrence of $x$ is preceded and followed by $y$.
\end{theorem}

We will also take the following related technical lemma from \cite{KHOO}.

\begin{lemma}\label{lemma:switch}
    If $w\in F_2$ is potentially positive and satisfies \goldC\ with $x=\Bb$ and $y=\ar$, applying $\bb \mapsto\bb \ar$ to $w$ sends it to a word satisfying \goldC\ either with $x=\Bb$ and $y=\ar$ or with $x=\Ar$ and $y=\bb$. By symmetry, if $w$ is potentially positive and satisfies \goldC\ with $x=\Ar$ and $y=\bb$, applying $\ar \mapsto \ar\bb$ to $w$ sends it to a word satisfying \goldC\ either with $x=\Bb$ and $y=\ar$ or with $x=\Ar$ and $y=\bb$.
\end{lemma}

Finally, for brevity, we will use the shorthand $\PP r$ to represent the set of cyclically-reduced potentially positive words in in $F_r$.

\section{Automata and matrices}
\label{ss:automata}

\subsection{Automata and matrices}

Our method for creating lower bounds will be to provide automata which generate potentially positive words (i.e.\ regular languages of potentially positive words) and calculate their growth from the automaton.
We will consider an automaton $A$ to be a directed graph whose nodes are labeled with letters in the free group (that is the generators and their inverses).
We will call this labeling function $\eta_A$.
$\eta_A$ induces a mapping from paths in the digraph to words in the free group.
A \textbf{path} here is defined to be a sequence (finite unless specified) of nodes such that for any two consecutive nodes $x$ and $y$, there is an edge from $x$ to $y$.
Considering the path as a sequence of nodes, we get a word by applying $\eta_A$ to each node in the sequence.

For this reason we will call an automaton \textbf{reduced} when no edge of the digraph connects a node to an inversely labeled node.
If an automaton is reduced then the image of any path is a reduced word.

We will in particular be interested in closed paths and their images.
We say a path from a node $x$ to a node $y$ is \textbf{closed} if there exists an edge from $y$ to $x$.
We will consider the empty path, that is the path with no nodes, to be closed by vacuity.
Closed paths in reduced automata must then lift to cyclically reduced words, since we require that there exists an edge between the last and first letters.

We will also define a \defnsb{mixing\footnote{We borrow this terminology from symbolic dynamics. See \cite{dynamicsbook}*{Def.\ 4.5.9} for a comparison.} automaton}.
\begin{definition}
We say a digraph (or by extension an automaton) is \defnsb{mixing} if there exists some $N$ such that for every ordered pair of nodes $x$ and $y$, and for every $n\geq N$, there exists a path from $x$ to $y$ of length $n$.
\end{definition}

These automata will be useful to us because of the following theorem which allows us to count closed paths:

\begin{lemma}\label{thm:irreducible eigenval}
  Let $G$ be a mixing sofic shift whose adjacency matrix has largest eigenvalue $\lambda$.
  The number of closed paths of length $n$ in $G$ is $\Theta(\lambda^n)$
\end{lemma}
\begin{proof}
  An irreducible sofic shift has a right resolving presentation $\mathcal{G}=(G',L)$ (\cite{dynamicsbook}*{Thm. 3.3.11}) and for a right resolving graph the number of paths $|L_n|$ of length $n$ satisfies $\frac{1}{k}|L_n(G')|\leq |L_n(\mathcal{G})|\leq |L_n(G')|$ (\cite{dynamicsbook}*{Thm. 4.1.13}), where $k$ is the number of states in $G'$. An irreducible graph has $C \lambda^n\leq |L_n|\leq C'\lambda^n$.
  The number of closed paths $P_{n}$ of length $n$ has $|P_{n}|\leq |L_n|\leq C'\lambda^n$.
  Since $G$ is mixing there is an $N$ so that every pair of nodes $v,v'$ has a path from $v$ to $v'$ of length $N$.
  Therefore there is an embedding of $L_n$ into $P_{n+2N,w}$ (the closed paths about the node w) by adding a path of length $N$ from $w$ to the word and a path of length $N$ from the word to $w$.
  Therefore $C \lambda^{-2N}\lambda^n \leq |L_{n-2N}|\leq |P_{n,w}|\leq |P_n|$ for all $n\geq 2N$.
\end{proof}
In practice nearly everything discussed here will be mixing, however for completeness we can give the following corollary:

\begin{corollary}\label{cor:eigenval}
  Let $G$ be the sum of finitely many mixing digraphs.
  The number of closed paths of length $n$ in $G$ is $\Theta(\lambda^n)$, where $\lambda$ is the largest eigenvalue of the adjacency matrix of $G$.
\end{corollary}
\begin{proof}
  The adjacency matrix of the sum of digraphs is a block matrix with blocks consisting of each component of the sum.
  $\lambda$ is then the maximum eigenvalue of the blocks.
  It follows then that if $G$ has $k$ components the number of closed paths is $O(k\lambda^n)=O(\lambda^n)$.
  
  Since there is also a component which has largest eigenvalue $\lambda$, the paths in that block grow $\Theta(\lambda^n)$.
  Thus the number of closed paths in $G$ is $\Omega(\lambda^n)$.
\end{proof}

Our strategy then is to use \autoref{cor:eigenval} to calculate the number of closed paths in the digraph.
If our labeling map is one-to-one, then this gives us the $\Theta$-class for the set of (cyclically reduced) words produced by the machine.
However, injectivity is going to be too strong of a requirement for what we want to do.

\begin{definition}\label{dfn:one-to-const}
  The node labeling of an automaton is \defnsb{one-to-constant} if there exists a constant $c$ such that for any cyclically reduced word $w$ there are at most $c$ closed paths whose image is $w$.
\end{definition}

To make it easier to determine if an automaton's labeling is one-to-constant we will give the following helpful lemma:

\begin{lemma}\label{lemma:one-to-poly-cond}
  The node labeling of a mixing automaton is one-to-constant iff no distinct paths have the same start-point, end-point, and image.
\end{lemma}
\begin{proof}
  Suppose we have two distinct paths, $p_0$ and $p_1$, from $x$ to $y$ with $\eta_A(p_0)=\eta_A(p_1)$.
  Now let us have an arbitrary path $q$ from $y$ to $x$.
  Since the automaton is mixing $q$ must exist.
  Let $p_0\circ q$ denote the path formed by following $p_0$ and then $q$.
  We have that $\eta_A(p_0\circ q) = \eta_A(p_1\circ q)$.
  Let us say $w = \eta_A(p_0\circ q)$.
  Since there are at least two ways to get from $x$ to itself with an image of $w$ there are at least $2^n$ ways to get from $x$ to itself with an image of $w^n$.
  Since the length $w^n$ grows linearly on $n$ and $2^n$ grows exponentially on $n$, there is no polynomial bound on the number of duplicate words produced, let alone a constant bound.

  Now let us suppose that no such pair of paths exist.
  For an arbitrary pair of closed paths which map to the same word, we can show they must be completely disjoint.
  If they were not disjoint they would share some node $x$.
  Then we could construct distinct paths from $x$ to itself which map to the same word by following our two closed paths starting from $x$.
  Since any collection of duplicate words must be induced by a set of pairwise disjoint closed paths, the total number of duplicate words is bounded above by the number of nodes, a constant.
\end{proof}

The following:

\begin{lemma}
  Let $A$ be a mixing automaton whose labeling is one-to-constant.
  Let $p(n)$ be the number of closed paths of length $n$ in $A$.
  Let $w(n)$ be the number of cyclic words of length $n$ produced by $A$.
  $w\in \Theta(p)$.
\end{lemma}
\begin{proof}
  Since the labeling is one-to-constant and preserves length, $\dfrac{1}{c}p(n)\leq w(n) \leq p(n)$.
  Therefore $w\in \Theta(p)$.
\end{proof}

From here we have everything that is needed to prove our lower bounds. 
We will give several machines, demonstrate they are reduced, mixing, one-to-constant and produce only potentially positive words.
Their eigenvalues then give us lower bounds.

However, we think that this would be a be a bit mystical.
Our techniques for producing the automaton may be useful.
It is possible for a fixed automaton $A$ and automorphism $\varphi$ to construct an automaton which produces exactly the images of the words produced by $A$ under $\varphi$.
Our strategy is then to start from the automaton which produces the positive words and apply automorphisms to increase the largest eigenvalue.
The automata we produce are all the result of limiting an infinite family of automorphisms.

\section{Lower bounds}

\subsection{Rank 2}
\begin{figure}[h]
\caption{A machine which only produces potentially positive words}
\begin{tikzpicture}
\begin{scope}[state,thick]
  \node[draw=ourred](s1){$\bigar$};
  \node[draw=ourblue](s2)[below=of s1]{$\bigBb$};
  \node[draw=ourred](s3)[right=of s2]{$\bigar$};
  \node[draw=ourblue](s4)[above=of s3]{$\bigbb$};
  \node[draw=ourred](s5)[right=of s4]{$\bigAr$};
\end{scope}
\path[->,thick]
  (s1) edge[ourred,loop left] ()
       edge[ourred] (s2)
       edge[ourred,bend left] (s4)
  (s2) edge[ourblue] (s3)
  (s3) edge[ourred,loop right] ()
       edge[ourred,] (s4)
  (s4) edge[ourblue,loop above] ()
       edge[ourblue,bend left] (s1)
       edge[ourblue,bend left] (s5)
  (s5) edge[ourred,loop right] ()
       edge[ourred,bend left] (s4)
  ;
\end{tikzpicture}
\label{machine:lb}
\end{figure}
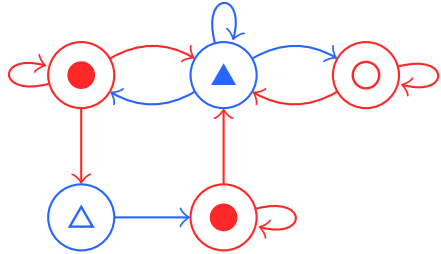

\begin{theorem}\label{thm:lowerboundF2}
  All words generated by the machine in \autoref{machine:lb} are potentially positive. Furthermore, the growth of paths of length $n$ through this machine is $\Theta(\lambda_1^n)$ where $\lambda_1 \approx 2.50506841362147$ is the largest root of the polynomial
  \begin{align*}
    \left(1-\lambda\right) \left(\lambda^4 - 3 \lambda^3 + \lambda^2 + \lambda - 1\right).
  \end{align*}
\end{theorem}
\begin{proof}
  Fix a word $w$ produced by the machine.
  Consider the effect of $\bb \mapsto \bb \ar$ on $w$.
  All the $\Ar$ syllables in $w$ are flanked by $\bb$ on both sides, so $\bb\mapsto\bb\ar$ decreases each of their lengths by $1$.
  Since each $\Bb$ is preceded by $\bb \ar^n$, any $\Ar$ it introduces cancels with one of the $\ar$s and the $\bb$ introduces a new $\ar$ to replace the one that was canceled.
  Thus the resulting word is still generated by the machine and has strictly fewer $\Ar$s (unless it had none to begin with).
  By induction, repeated application of this automorphism eventually results in a word with no $\Ar$s.
  Since all words in $F_2$ which have only positive exponents on one of the generators are potentially positive, this word is potentially positive.

  The machine has two nodes labeled with $\ar$, however we can be sure there are no duplicate paths since one has only incoming edges from itself and $\bb$ and the other has only incoming from itself and $\Bb$.
  
  The characteristic polynomial of the transition matrix for the machine in \autoref{machine:lb} is
  \begin{align*}
    \left(1-\lambda\right) \left(\lambda^4 - 3 \lambda^3 + \lambda^2 + \lambda - 1\right).
  \end{align*}
  This has largest root $\lambda_1 \approx 2.50506841362147$, and thus the growth of paths of length $n$ is $\Theta(\lambda_1^n)$.
\end{proof}

\subsection{Rank 3 through 7}
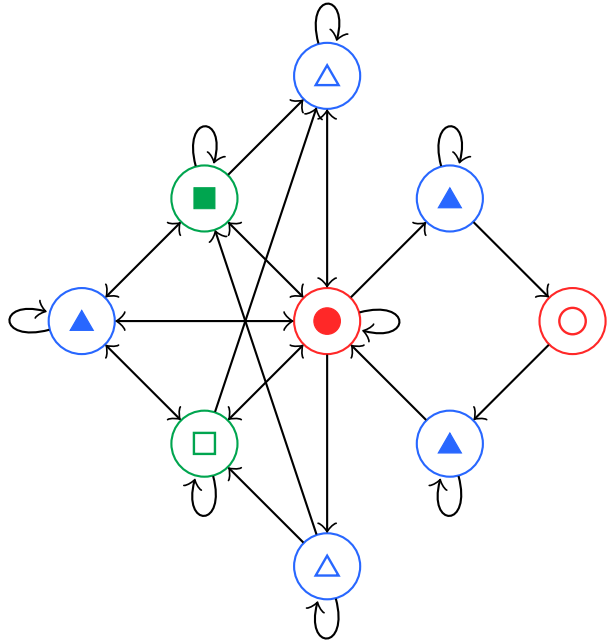
\begin{figure}[h]
\begin{tikzpicture}
\begin{scope}[state,thick]
  \node[draw=ourred](s1){$\bigar$};
  \node[draw=ourblue](s2)[above right=of s1]{$\bigbb$};
  \node[draw=ourred](s3)[below right=of s2]{$\bigAr$};
  \node[draw=ourblue](s4)[below left=of s3]{$\bigbb$};
  \node[draw=ourgreen](s5)[above left=of s1]{$\bigcg$};
  \node[draw=ourgreen](s6)[below left=of s1]{$\bigCg$};
  \node[draw=ourblue](s7)[below left=of s5]{$\bigbb$};
  \node[draw=ourblue](s8)[above right=of s5]{$\bigBb$};
  \node[draw=ourblue](s9)[below right=of s6]{$\bigBb$};
\end{scope}
\path[->,thick]
    (s1) edge[loop right] ()
         edge[] (s2)
         edge[] (s9)
    (s2) edge[loop above] ()
         edge[] (s3)
    (s3) edge[] (s4)
    (s4) edge[loop below] ()
         edge[] (s1)
    (s5) edge[loop above] ()
         edge[] (s8)
    (s6) edge[loop below] ()
         edge[] (s8)
    (s7) edge[loop left] ()
    (s8) edge[loop above] ()
    (s9) edge[loop below] ()
         edge[] (s5)
         edge[] (s6)
  ;
  \path[<->,thick]
    (s1) edge (s5)
         edge (s6)
         edge (s7)
         edge (s8)
    (s5) edge (s7)
    (s6) edge (s7)
  ;
\end{tikzpicture}
\caption{This machine only produces potentially positive words.}
\label{machine:lb3}
\end{figure}
\begin{figure}
    \centering
\begin{tikzpicture}
  \clip(-4,-5)rectangle(9.25,5.5);
  \begin{scope}[state,thick]
    \node[draw=ourred](s1){$\bigar$};
    \node[draw=ourblue](s2)[above right=of s1]{$\bigBb$};
    \node[draw=ourblue](s3)[below right=of s1]{$\bigBb$};
    \node[draw=ourgreen](s4)[right=of s2]{$\bigCg$};
    \node[draw=ourgreen](s5)[right=of s3]{$\bigCg$};
    \node[draw=ourblue](s6)[above left=of s1] {$\bigbb$};
    \node[draw=ourblue](s7)[below left=of s1] {$\bigbb$};
    \node[draw=ourred](s8)[below left=of s6] {$\bigAr$};
    \node[draw=magenta](s9)[below right=of s4] {$\bigdm^{\magenta \pm}$};
    \node[draw=ourblue](s10)[above right=of s9] {$\bigbb$};
    \node[draw=ourgreen](s11)[below right=of s9] {$\bigcg$};
    \node[draw=ourgreen](s12)[above left=of s10] {$\bigcg$};
    \node[draw=ourgreen](s13)[above right=of s10] {$\bigcg$};
    \node[draw=ourblue](s14)[above right=of s12] {$\bigBb$};
  \end{scope}
  \path[->,thick]
    (s1) edge[loop left,draw=ourred] (s1)
         edge[draw=ourred,bend left] (s2)
         edge[draw=ourred,bend left] (s3)
         edge[draw=ourred,out=90,in=120,looseness=1.5] (s4)
         edge[draw=ourred,out=-90,in=-120,looseness=1.5] (s5)
         edge[draw=ourred,bend right] (s6)
         edge[draw=ourred,bend left,out=-15,in=190] (s9)
         edge[draw=ourred,out=90,in=90,looseness=1.6] (s10)
         edge[draw=ourred,out=-90,in=-120,looseness=1.8] (s11)
         edge[draw=ourred,out=90,in=130,looseness=1.2] (s12)

    (s2) edge[loop above,draw=ourblue] (s2)
         edge[draw=ourblue] (s4)
         edge[draw=ourblue] (s5)
         edge[draw=ourblue,out=-60,in=150,looseness=1] (s9)
         edge[draw=ourblue,out=60,in=195,looseness=1.2] (s12)
         edge[draw=ourblue,out=-60,in=140,looseness=0.5] (s11)

    (s3) edge[loop below,draw=ourblue] (s3)
         edge[draw=ourblue,bend left] (s1)

    (s4) edge[loop above,draw=ourgreen] (s4)
         edge[bend left,draw=ourgreen] (s9)
         edge[draw=ourgreen,out=-15,in=-75] (s10)

    (s5) edge[loop below,draw=ourgreen] (s5)
         edge[draw=ourgreen] (s3)
         edge[draw=ourgreen,in=-27.5,out=180,looseness=0.6] (s1)

    (s6) edge[loop below,draw=ourblue] (s6)
         edge[bend right,draw=ourblue] (s8)

    (s7) edge[loop above,draw=ourblue] (s7)
         edge[bend right,draw=ourblue] (s1)

    (s8) edge[bend right,draw=ourred] (s7)

    (s9) edge[draw=magenta,loop right] (s9)
         edge[draw=magenta,in=60,out=-110,looseness=1] (s3)
         edge[draw=magenta,bend left] (s5)
         edge[draw=magenta,out=170,in=10] (s1)
         edge[draw=magenta,out=15,in=-100] (s10)
         edge[draw=magenta,out=15,in=100,looseness=1.3] (s11)

    (s10) edge[draw=ourblue,loop right] (s10)
          edge[draw=ourblue,bend left] (s12)
          edge[draw=ourblue,out=-160,in=27.5] (s1)
          edge[draw=ourblue,out=-160,in=0] (s4)
          edge[draw=ourblue,out=-160,in=60] (s9)
          edge[draw=ourblue,out=-160,in=120] (s5)
          edge[draw=ourblue,out=-160,in=140,looseness=0.8] (s11)

    (s11) edge[loop right,draw=ourgreen] (s11)
          edge[draw=ourgreen,out=165,in=-35,looseness=1.4] (s3)
          edge[draw=ourgreen,in=-27.5,out=165] (s1)
          edge[draw=ourgreen,bend left] (s9)
          edge[draw=ourgreen,out=165,in=-75] (s10)

    (s12) edge[draw=ourgreen,loop right] (s12)
          edge[draw=ourgreen,bend left] (s14)

    (s13) edge[draw=ourgreen,loop left] (s13)
          edge[draw=ourgreen,bend left] (s10)
          edge[draw=ourgreen,out=195,in=80,looseness=0.2] (s3)
          edge[draw=ourgreen,out=195,in=-27.5,looseness=0.5] (s1)

    (s14) edge[draw=ourblue,bend left] (s13)
    ;
\end{tikzpicture}
    \caption{%
    The machine only producing potentially positive words described in (\ref{machine:lbj}) for rank 4.
    In that machine the nodes labeled $\dm$ and $\Dm$ have the exact same incoming and outgoing edges, except in that they each have a self-loop but no edge to the other.
    To reduce clutter, we have combined these two nodes with nearly identical behavior into a single node.
    This node is labeled $\dm^\pm$, and indicates that when following an incoming arrow to that node from another one, the machine can produce either $\dm$ or $\Dm$.
    However, when following an in arrow from $\dm^\pm$ to itself, the machine cannot switch the sign.
    When calculating the eigenvalue one can use the original matrix or the above if the edges directed to the $\dm^\pm$ labeled node (other than the self-loop) are weighted with a 2 (representing the two possibilities).
    }%
   \label{machine:lb4}
\end{figure}
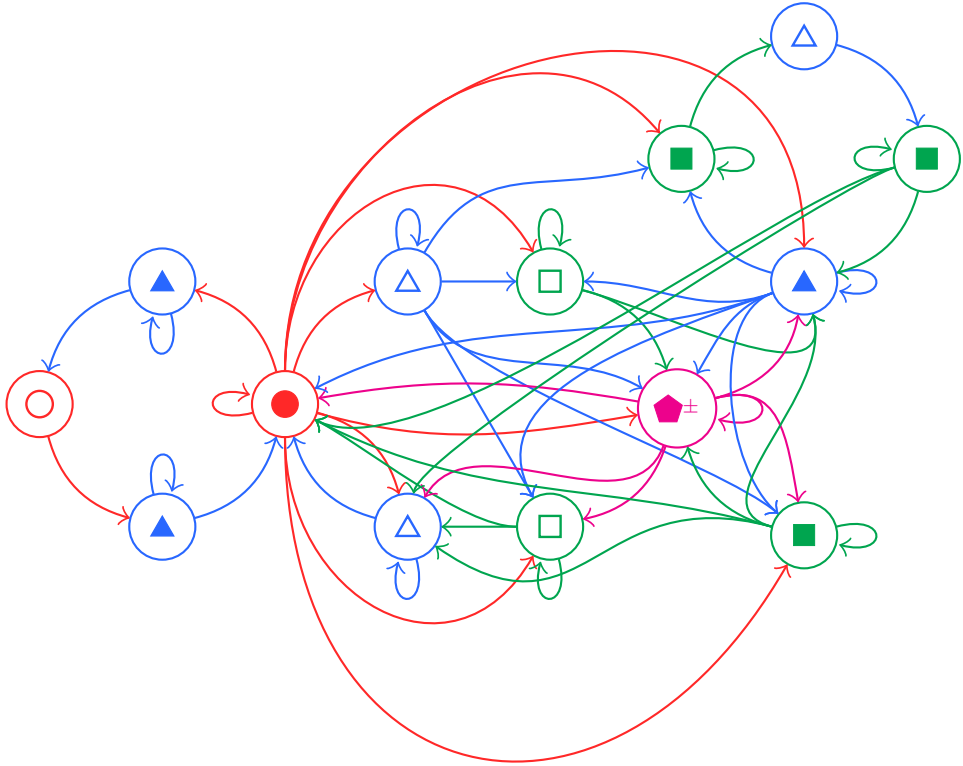
To get lower bounds in higher ranks we will provide a general strategy.
For this it might help to refer to the machines in Figures \ref{machine:lb3} and \ref{machine:lb4}.

Looking specifically at \autoref{machine:lb4},
the strategy to take words produced by this to positive words is to send out $\bb$ in both directions from $\ar$, then send out $\cg$ in both directions from $\bb$.
The circular portions of the machine on the left and right insulate the $\Ar$s and $\Bb$s until the $\Bb$s and $\Cg$s in the middle are eliminated.
This works because everything that can point to the $\cg$ nodes on the right eventually ``becomes'' a $\bb$ when we send $\bb$ out in both directions from $\ar$.
At the end we can use the $\cg$s to eliminate the remaining appearances of $\Bb$ and the $\bb$s to eliminate the remaining $\Ar$s.
Then we can use \autoref{prop:allbutone} to finish.
In higher ranks we iterate this strategy.
We will give the general description in more detail below.

We first present the general adjacency matrix of our automaton for $F_r$ as a block matrix \hyperref[machine:lbj]{below}.
Here $U$ is a matrix with all entries above the diagonal equal to 1 and all entries below and on the diagonal equal to 0.
$L$ is $U^{\intercal}$, and $I$ is the identity matrix.
All of these blocks are square $(r-2)\times (r-2)$ matrices.

We will also name and label each node.
We name the nodes with Greek letters corresponding to the row of the block array and with a subscript equal to the subscript of the label (e.g.\ there is no $\varepsilon_1$ node because there is no node labeled $x_1^{-1}$ in the $\varepsilon$ row of the block matrix).

It is helpful to note that in all cases a $U$ between two sets of nodes means that there is a transition between nodes where the source has a strictly higher index than the destination.
$L$ has the opposite meaning, with the source having a strictly lower index than the destination.

\begin{equation}
\label{machine:lbj}
\begin{blockarray}{rl[ccccccccc]}
\alpha_1 & \left\{x_1\right.
& 1 & 1 & 1 & 1 & 0 & 1 & 0 & 0 & 1 \\
\beta_i & \left\{\begin{matrix}x_2 \\ \vdots \\ x_{r-1}\end{matrix}\right.
& 1 & 1 & 1 & 1 &U+L& U & 0 & 0 & U \\
\gamma_r & \left\{x_r\right.
& 1 & 1 & 1 & 0 & 1 & 0 & 0 & 0 & 0 \\
\delta_r & \left\{x_r^{-1}\right.
& 1 & 1 & 0 & 1 & 1 & 0 & 0 & 0 & 0 \\
\varepsilon_i & \left\{\begin{matrix} x_2^{-1} \\ \vdots \\ x_{r-1}^{-1} \end{matrix}\right.
& 1 & L & 0 & 0 &L+I& 0 & 0 & 0 & 0 \\
\zeta_i & \left\{\begin{matrix} x_2^{-1} \\ \vdots \\ x_{r-1}^{-1} \end{matrix}\right.
& 1 &U+L& 1 & 1 & L &U+I& 0 & 0 & U \\
\eta_i & \left\{\begin{matrix} x_2 \\ \vdots \\ x_{r-1} \end{matrix}\right.
& 1 & L & 0 & 0 & L & 0 & I & 0 & 0 \\
\theta_i & \left\{\begin{matrix} x_1^{-1} \\ \vdots \\ x_{r-2}^{-1} \end{matrix}\right.
& 0 & 0 & 0 & 0 & 0 & 0 & I & 0 & 0 \\
\iota_i & \left\{\begin{matrix} x_2 \\ \vdots \\ x_{r-1} \end{matrix}\right.
& 0 & 0 & 0 & 0 & 0 & 0 & 0 & I & I \\
\end{blockarray}
\end{equation}

\begin{theorem}
  All words produced by (\ref{machine:lbj}) are potentially positive.
\end{theorem}
\begin{proof}
    Let us have a word $w$ produced by the machine in (\ref{machine:lbj}).
    We will explicitly construct a sequence of automorphisms which brings $w$ to a positive word.
    To begin we will apply the automorphism $x_1\mapsto x_2^{n_2}x_1x_2^{n_2}$ where $n_2$ is one more than the exponent on the largest $x_2^{-1}$ syllable, i.e.\ the length of the longest run of consecutive $x_2^{-1}$s.
    Three nodes in the machine are labeled with $x_2^{-1}$.
    Note that two of these, $\varepsilon_2$ and $\zeta_2$, must always occur in a syllable adjacent to a $x_1$.
    Specifically $\varepsilon_2$ must always come after a $x_1$ and $\zeta_2$ must always come before an $x_1$.
    Thus applying the automorphism $x_1\mapsto x_2^{n_2}x_1x_2^{n_2}$ replaces all syllables generated by $\varepsilon_2$ and $\zeta_2$ with positive $x_2$ syllables.
    Note further that the only place $x_1^{-1}$ can occur is in substrings of the form $x_1x_2^ix_1^{-1}x_2^jx_1$ (for positive $i$ and $j$).
    Thus the $x_2^{-n_2}$ syllables introduced from $x_1^{-1}$ syllables by applying this automorphism are immediately canceled.

    Next we apply $x_2\mapsto x_3^{n_3}x_2x_3^{n_3}$ where $n_3$ is one more than the exponent on the largest $x_3^{-1}$ syllable.
    This eliminates all the $x_3^{-1}$ syllables produced by $\varepsilon_3$ and $\zeta_3$.
    Note that $x_2^{-1}$ syllables produced by $\theta_i$ may appear in various substrings in the initial word $w$, but after applying $x_1\mapsto x_2^{n_2}x_1x_2^{n_2}$, they always appear in substrings of the form $x_2x_3^ix_2^{-1}x_3^jx_2$.
    Thus no new $x_3^{-1}$ syllables are introduced by applying this second automorphism.

    We continue this process by having each generator up to $x_{r-2}$ produce the following generator in both directions, until we have eliminated all syllables originating from $\varepsilon_i$ and $\zeta_i$.

    At this point every negative syllable is either:
    \begin{enumerate}
      \item \label{item:ears} appearing in a subword of the form $x_{r-1}x_i^{-1}x_{r-1}$ for $i < r-1$.
      (These originate from $\theta_i$)
      \item a $x_r^{-1}$ syllable.
      (These originate from $\delta_i$)
    \end{enumerate}
    We will handle (\ref{item:ears}) first.
    Note that at this point all $x_{r-1}$ syllables are positive.
    We start by applying the automorphism $x_{r-1} \mapsto x_{r-1}x_{r-2}$.
    This eliminates all the remaining negative $x_{r-2}$ syllables since they are all adjacent to a $x_{r-1}$.
    Now every syllable previously covered by (\ref{item:ears}) is adjacent to a $x_{r-2}$, so we perform the automorphism $x_{r-2}\mapsto x_{r-2}x_{r-3}$ and continue in this descending pattern until we apply $x_2\mapsto x_2x_1$, eliminating the final negative syllable of type (\ref{item:ears}).

    Now the only negative syllables remaining are of the form $x_r^{-k}$, so by \autoref{prop:allbutone} we are done.
    
\end{proof}

\section{Upper bound}

We want to produce an upper bound for the growth rate of potentially positive words in $F_2$, so we will begin by using a necessary condition for potential positivity.

\begin{figure}[h]
\caption{A machine which produces exactly those words words which satisfy \goldC\ with $y=\Bb$ and $x=\ar$\\}
\begin{tikzpicture}
\begin{scope}[state,thick]
  \node[draw=ourred](s1){$\bigAr$};
  \node[draw=ourblue](s2)[right=of s1]{$\bigbb$};
  \node[draw=ourred](s3)[right=of s2]{$\bigar$};
  \node[draw=ourblue](s4)[right=of s3]{$\bigBb$};
\end{scope}
\path[->,thick]
  (s1) edge[ourred,loop below] ()
       edge[ourred,bend left] (s2)
  (s2) edge[ourblue,loop below] ()
       edge[ourblue,bend left] (s1)
       edge[ourblue,bend left] (s3)
  (s3) edge[ourred,loop below] ()
       edge[ourred,bend left] (s2)
       edge[ourred,bend left] (s4)
  (s4) edge[ourblue,bend left] (s3)
  ;
\end{tikzpicture}
\label{mcm:gold}
\end{figure}

The machine in \autoref{mcm:gold} produces all words which satisfy \goldC\ for a particular choice of $x$ and $y$.
It is easy to see that the growth rate of all words which satisfy \goldC\ has the same growth rate, since there are 8 meaningful $x,y$ combinations and each combination gives a relabeling of the machine in \autoref{mcm:gold}.

Thus we have the following:

\begin{lemma}\label{lemma:goldgrowth}
  The growth of $\PP2$ is $\Omega(\lambda_1^n)$, where $\lambda_1\approx 2.53209$, is the largest root of the polynomial:
  \begin{align}\label{eqn:goldcp}
  \lambda^3-3\lambda^2+3
  \end{align}
\end{lemma}
\begin{proof}
  From \autoref{mcm:gold} we get the following transition matrix:
  \begin{align}\label{eqn:goldmatrix}
  \begin{bmatrix}
  1 & 1 & 0 & 0\\
  1 & 1 & 1 & 0\\
  0 & 1 & 1 & 1\\
  0 & 0 & 1 & 0
  \end{bmatrix}
  \end{align}
  We can calculate the characteristic polynomial of (\ref{eqn:goldmatrix}) to be
  \begin{equation*}
    \lambda\left(\lambda^3 -3\lambda^2+3\right)
  \end{equation*}
  
  Since each node has a unique label, the function assigning paths to words is certainly injective.
  Thus the growth rate of the machine in \autoref{mcm:gold} is the largest root of (\ref{eqn:goldcp}).
  By \autoref{gold}, every potentially positive word is generated by this machine, so the growth rate of potentially positive words is $\Omega(\lambda_1^n)$.
\end{proof}

We will use the notation $\gamma$ to mean the automorphism $\ar \leftrightarrow \bb$.
We have seen that the number of potentially positive words is bounded above by the number of words that satisfies \goldC.
We can bring this bound down by considering words which satisfy \goldC, and still satisfy it after applying an automorphism.
By \autoref{lemma:switch}, if a potentially positive word satisfies the criterion with a certain $x$ and $y$, after applying a certain automorphism it must either satisfy the criterion with the same $x$ and $y$ or with $\gamma(x)$ and $\gamma(y)$.
We can then split the set of potentially positive words satisfying \goldC\ after applying the automorphism from \autoref{lemma:switch} into two sets; one corresponding to $x$ and $y$, the other to $\gamma(x)$ and $\gamma(y)$.
The $\Theta$-class of words satisfying the criterion after the automorphism is then the $\Theta$-class of the larger of these two sets, so we can consider them separately.
Furthermore, if we apply the automorphism specified by \autoref{lemma:switch} again, each of these two sets can be split into two more cases in the same manner. 
Iterating this process gives an infinite binary tree (\autoref{figure:tree}).
Each layer of the tree consists of all the ways a word can satisfy \goldC\ over the course of a fixed number of iterations of this process.
The fastest growing set in each layer gives an upper bound on the number of potentially positive words in $F_2$. We formalize this idea with the following definition.

\begin{definition}\label{def:tree}
  Let us define $G$ to be the set of words satisfying \goldC\ with $x=\Bb$ and $y=\ar$.
  Let $\varphi$ be the automorphism $\bb \mapsto \bb \ar$ described in \autoref{lemma:switch} and let $\gamma$ be the automorphism $\bb \leftrightarrow \ar$ as before.
  Recursively:
  \begin{align*}
  R(S) &= \left\{w: w\in G, \varphi(w)\in S\right\} \\
  L(S) &= \left\{w: w\in G, \gamma(\varphi(w))\in S\right\}
  \end{align*}
  We will employ a convenient shorthand wherein we leave the parentheses and a final $G$ as implicit.
  Thus, e.g. we will write $L^2R$ to mean $L^2(R(G))$.
\end{definition}

\begin{figure}
  \begin{tikzpicture}
    \node(e) at (4,0) {$G$};
    
    \node(l) at (2,-2) {$L$};
    \node(r) at (6,-2) {$R$};

    \node(ll) at (1,-4) {$L^2$};
    \node(lr) at (3,-4) {$LR$};
    \node(rl) at (5,-4) {$RL$};
    \node(rr) at (7,-4) {$R^2$};

    \node(lll) at (1/2,-6) {$L^3$};
    \node(llr) at (3/2,-6) {$L^2R$};
    \node(lrl) at (5/2,-6) {$LRL$};
    \node(lrr) at (7/2,-6) {$LR^2$};

    \node(rll) at (9/2,-6) {$RL^2$};
    \node(rlr) at (11/2,-6) {$RLR$};
    \node(rrl) at (13/2,-6) {$R^2L$};
    \node(rrr) at (15/2,-6) {$R^3$};
    
    \path
      (e) edge (l)
          edge (r)
      (l) edge (ll)
          edge (lr)
      (r) edge (rl)
          edge (rr)
      (ll) edge (lll)
           edge (llr)
      (lr) edge (lrl)
           edge (lrr)
      (rl) edge (rll)
           edge (rlr)
      (rr) edge (rrl)
           edge (rrr)
    ;
  \end{tikzpicture}
  \caption{The first four layers of sets drawn as a tree}
  \label{figure:tree}
\end{figure}
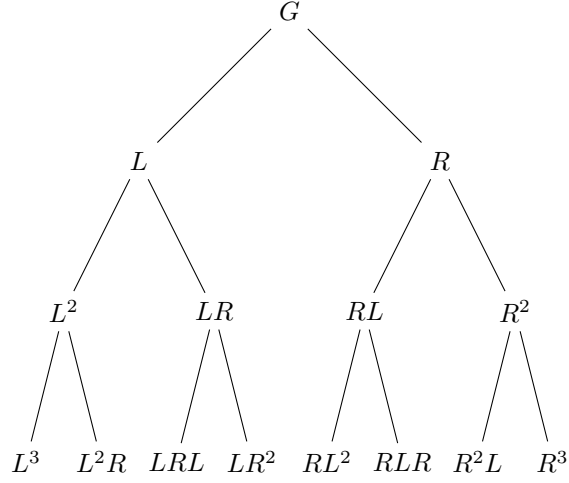

Our proof of \autoref{thm:main} relies on the following three facts about the tree in \autoref{figure:tree}.

\begin{proposition}\label{prop:children are small}
  For every set $S$ in the tree, $\mathfrak g(SL), \mathfrak g(SR)\in O\left(\mathfrak g(S)\right)$.
  In other words, the growth rate of the children of $S$ is at most that of $S$.
\end{proposition}

We will call the set of words generated by the machine in \autoref{machine:lb} $R^\infty$, a notation which will make sense in light of the following proposition:
\begin{proposition}\label{prop:Rlim}
\begin{equation*}
\bigcap^\infty_{n=0} R^n = R^\infty
\end{equation*}
\end{proposition}

\begin{proposition}\label{prop:dont go left}
For all $n \in \mathbb N$,
\begin{align*}
   \mathfrak{g}(R^nL)\in o(\mathfrak{g}(R^\infty))
\end{align*}
\end{proposition}

Taking these propositions as given for now, we can prove our main result:
\begin{proof}[Proof of \autoref{thm:main}]
  Let us call the set of words produced by the $n$th layer of the tree $Y_n$.
  These can be thought of as all words which do not reach a non-positive word after $n$ steps of Goldstein's algorithm.

  We can see that $\mathfrak g(Y_n)$ gives successive upper bounds on the growth of $\PP 2$ which get tighter as $n$ approaches infinity.
  We will show that $\mathfrak g(Y_n)\sim \mathfrak g(R^n)$, and so the pointwise limit of $\mathfrak g(R^n)$ as $n$ approaches infinity is an upper bound.

  Since $R^nL$ grows strictly slower than $R^\infty$ (\autoref{prop:dont go left}), by induction, all sets in the tree which include any $L$ in their name must have strictly slower growth than $R^\infty$.
  Since $R^\infty$ is a subset of $R^n$ it must grow slower ($\mathfrak g(R^n)\in \Omega(\mathfrak g(R^\infty))$), and so for any node in the tree $S$ containing an $L$, $\mathfrak g(S)\in o(\mathfrak g(R^n))$.
  Since $Y_n$ is the union of $R^n$ with other sets in the tree containing $L$, $\mathfrak g(Y_n)\sim \mathfrak g(R^n)$.

  Since $\mathfrak g(R^n)$ converges to at most $\mathfrak g(R^\infty)$, this means our upper bounds converge to $O(\mathfrak g_k(R^\infty)) = O(\mathfrak \lambda_1^k)$.
  From \autoref{thm:lowerboundF2} we have that $\mathfrak g_k(\PP2) \in \Omega(\lambda_1^k)$ and so $\mathfrak g_k(\PP2) \in \Theta(\lambda_1^k)$.
\end{proof}

All that remains is to prove the propositions.
The proof of \autoref{prop:children are small} is the shortest.
\begin{proof}[Proof of \autoref{prop:children are small}]
  We will prove only that $\mathfrak g(SL) \in O(\mathfrak{g}(S))$.
  The other case follows in the same manner.
  We proceed by induction over all paths in the tree.
  By \autoref{def:tree}, we have that $L \subseteq G$.
  Let $Z_0$ and $Z_1$ be sets such that $Z_0 \subseteq Z_1$.
  It follows then that $R(Z_0)\subseteq R(Z_1)$ and likewise $L(Z_0)\subseteq L(Z_1)$.
  Then by induction $SL \subseteq S$ and therefore $\mathfrak g(SL) \in O(\mathfrak{g}(S))$.
\end{proof}

Now we will give descriptions of the languages $R^n$ and $R^nL$ in terms of forbidden subwords (called ``forbidden blocks'' in \cite{dynamicsbook}).
This description will be useful for proving the other two propositions.

\begin{lemma}\label{lem:Rn}
The set $R^n$ is all reduced words that do not contain the following subwords:
\begin{itemize}
\item $\Bb \Ar$
\item $\Ar\Bb$
\item all $\Bb \ar^k\Bb$ where $0\leq k\leq n$
\end{itemize}
\end{lemma}
\begin{proof}
  First we will show none of these subwords may appear.
  It follows from \goldC\ that $\Bb \Ar$, $\Ar\Bb$, and $\Bb \Bb$ cannot appear.
  The set $R^n$ can be given as:
  \begin{align}
  R^n = \left\{w \mid w\in G, \left(\bb \mapsto \bb \ar^n\right)(w) \in G\right\}
  \end{align}
  Suppose $w\in R^n$ contains a subword $\Bb \ar^k\Bb$.
  Its image under $\bb \mapsto \bb \ar^n$ must contain $\Bb \ar^{k-n}\Bb$, since the right $\Bb$ produces $\ar^{-n}$ on its left and neither $\Bb$ can cancel under an automorphism which only introduces new $\ar$ and $\Ar$ letters.
  If $k < n$, this image then contains $\Bb \Ar$ which is forbidden in $G$.
  If $k = n$, this image then contains $\Bb \Bb$ which is also forbidden in $G$.
  Therefore $k$ must be larger than $n$.
   
  Suppose $w$ is a cyclic word not containing $\Bb \Ar$, $\Ar\Bb$, or $\Bb \Bb$.
  Then $\Bb$ can only appear next to $\ar$ in $w$, so $w\in G=R^0$. We will proceed by induction on $n$.
  Suppose $w$ does not contain $\Bb \Ar$, $\Ar\Bb$, or $\Bb \ar^k\Bb$ for any $0\leq k \leq n$.
  We will show that if $\varphi$ is the automorphism $\bb \mapsto \bb \ar$, $\varphi(w)$ does not contain $\Bb \Ar$, $\Ar\Bb$, or $\Bb \ar^k\Bb$ for any $0\leq k \leq n-1$.
  Suppose by way of contradiction that it does contain one of these types of subword.
  If $\varphi(w)$ contains $\Ar\Bb$, $\varphi$ must have added an $\Ar$ that did not cancel, so $w$ must have had a subword of the form $\Ar\Bb$ or $\Bb \Bb$, contradicting our assumption.
  Similarly, if $\varphi(w)$ contains $\Bb \Ar$, either the $\Ar$ was already there in $w$ or it was added by $\varphi$. Therefore $w$ either contained a subword of the form $\Bb \Ar$ or $\Bb \Bb$, contradicting our assumption.
  Finally, if $\varphi(w)$ contains a subword of the form $\bb \ar^k\bb$ for $0\leq k \leq n-1$, $w$ must have contained a subword of the form $\bb \ar^{k+1}\bb$, again contradicting our assumption.
\end{proof}

\begin{lemma}\label{lem:RnL}
The set $R^nL$ is all reduced words in $R^n$ which do not contain any subwords of the form $\bb \Ar^k\bb$ where $k > n + 2$.
\end{lemma}
\begin{proof}
  Let $\varphi$ be the automorphism $\bb \mapsto \bb \ar$.
  First, let us note that the set $R^nL$ can be given as
  \begin{equation*}
    \begin{split}
      R^nL
        &= \{ w \mid w \in G, \varphi^n(w) \in L\} \\
        &= \{ w \mid w \in G, \varphi^n(w) \in G, \gamma(\varphi^{n+1}(w)) \in G \}.
    \end{split}
  \end{equation*}
  This shows immediately that if $w$ is not in $R^n$ then it is not in $R^nL$.
  
  Next, suppose $w$ contains a subword of the form $\bb \Ar^k\bb$ for some $k > n + 2$.
  Let $\psi$ be the automorphism $\gamma\circ\varphi^{n+1}$.
  Then $\psi(\bb \Ar^k\bb)$ contains $\ar\Bb^{k - n - 1}\ar$ as a subword.
  Since $k > n + 2$, the exponent on $\Bb$ in the image of the subword must be at least two, which means that $\psi(w)$ contains $\Bb \Bb$ as a subword and so is not in $G$.
  Therefore if $w$ contains a subword of the form $\bb \Ar^k\bb$, it is not in $R^nL$.

  Suppose, on the other hand, that $w \in R^n$ and $w$ does not contain a subword of the form $\bb \Ar^k\bb$ where $k > n + 2$.
  To show that $w$ is in $R^nL$, all we need to do is show that $\gamma(\varphi^{n+1}(w)) \in G$.
  This is equivalent to showing that $\Ar\Bb$, $\Bb \Ar$, and $\Ar\Ar$ do not appear as subwords in $\varphi^{n+1}(w)$.
  Any occurrence of $\Ar$ in $\varphi^{n+1}(w)$ must have already been in $w$.
  Indeed, any new $\Ar$ must be produced from an $\Bb$, however each $\Bb$ must be adjacent on both sides to $\ar$ syllables, and by \autoref{lem:Rn} the exponent on these syllables must be greater than $n$.
  Since $w \in G$, it cannot contain $\Bb \Ar^m\bb$ or $\bb \Ar^m\Bb$ as subwords, and so any occurrence of $\Ar$ in $w$ must be in a subword of the form $\bb \Ar^k\bb$ for $k > 0$.
  Thus, so long as $k \leq n + 2$ for any such subwords of $w$, $\Ar\Ar$ cannot be a subword of $\varphi^{n+1}(w)$, and so $w$ is in $R^nL$.
\end{proof}

\begin{proof}[Proof of \autoref{prop:Rlim}]
It follows from \autoref{lem:Rn} that the union of all $R^n$ are the words not containing $\Ar\Bb$, $\Bb \Ar$, or $\Bb \ar^k\Bb$ for any $k \geq 0$.

It is clear from \autoref{machine:lb} that none of these subwords can appear in $R^\infty$.

Suppose $w$ does not contain $\Ar\Bb$, $\Bb \Ar$, or $\Bb \ar^k\Bb$ for any $k \geq 0$. We will show that the machine in \autoref{machine:lb} produces $w$. From our assumptions it follows that every $\Bb$ in $w$ is preceded and followed by $\ar$ and that $\Ar$ cannot appear before or after $\Bb$. Finally, a $\ar$ that occurs after an $\Bb$ cannot be followed by an $\Bb$. Thus we see that indeed $w$ is a path in the machine.
\end{proof}

\begin{proof}[Proof of \autoref{prop:dont go left}]
Let us consider the cyclic words in $R^nL$.
We will define a function $f$ which takes a generic subset of $R^nL$ and assigns each word to a word of equal length in $R^\infty$.
This will show that $\mathfrak g_k\left(R^nL\right)\in O\left(\mathfrak g_k\left(R^\infty\right)\right)$, because such an encoding would show that the chosen subset grows as $R^\infty$ and the remainder is negligible and so does not effect the asymptotic growth.

There are two differences between the machines $R^nL$ and $R^\infty$:
\begin{itemize}
\item $R^nL$ forbids all subwords of the form $\bb \Ar^{n+2+k}\bb$ where $k$ is a positive integer, while $R^\infty$ does not.
\item $R^\infty$ forbids all subwords of the form $\Bb \ar^{n+k}\Bb$ where $k$ is a positive integer, while $R^nL$ does not.
\end{itemize}
All other subwords (not containing any of the above subwords) are either illegal in both machines, or legal in both machines.

Our function $f$ will then replace segments of the form:
\begin{align*}
\bb \ar^{k_0}\Bb \ar^{n+k_1} \Bb \ar^{n+k_2}\Bb \dots\Bb \ar^{n+k_{s-1}}\Bb \ar^{k_s} \bb
\end{align*}
where $s\geq2$ and $k_0, \dots, k_s$ are positive integers, with the following:
\begin{align*}
\bb \ar^{k_0-1} \bb \Ar^{n+k_1+2} \bb \Ar^{n+k_2} \bb \dots \bb \Ar^{n+k_{s-1}}\bb \ar^{k_s-1}\bb
\end{align*}
We present some observations.
\begin{itemize}
\item The segments we are replacing are legal in $R^nL$ but illegal in $R^\infty$.
\item It is impossible for segments of the first form to overlap, and thus there is no ambiguity with the replacement scheme.
\item So long as a word in $R^nL$ contains $\bb$, all its subwords that are illegal in $R^\infty$ must be of this form.
\item The image after replacement in $f$ is a string which is legal in $R^\infty$, but illegal in $R^nL$.
\item The segment and its image have the same length. 
\end{itemize}
If we consider the subset of words in $R^nL$ which do not contain $\bb$, we can see that they must also not contain $\Ar$, and thus they can grow at most $O\left(2^\ell\right)$.
This makes them negligible in the total set $R^nL$, which certainly grows faster than that.
Therefore, if we let $R^nL'$ be the set of words which are in $R^nL$ and contain $\bb$, we have that the growth of $R^nL'$ is asymptotically equivalent to $R^nL$.

The function $f : R^nL' \rightarrow R^\infty$ is an injective function which maintains the length of the input.
Thus we have

\begin{align*}
\growth\left(R^nL\right) \sim \growth\left(R^nL'\right)\in  O(\growth\left(R^\infty\right))
\end{align*}

Furthermore the subword $v=\ar\bb\Ar^n\bb\Ar^{n+3}\bb$ cannot appear in the image of any $w\in R^nL'$ under $f$.
This is because $v$ is illegal in $R^nL$, but cannot appear as the subword of a substituted segment since the first $\Ar$ syllable of such a segment always has exponent $> n+2$, and each other $\Ar$ syllable must be preceded by $\Ar\bb$ (while this is preceded by $\ar\bb$).

Since $R^\infty$ is a language produced by a irreducible mixing automaton, it is ``growth sensitive'' (see \cite{CSW}) and thus excluding any subword gives an exponentially negligible subset.
Thus $\growth_k\left(R^nL'\right)\in O(\varepsilon^k\growth_k(R^\infty))$ for some $\varepsilon < 1$, and so
\begin{align*}
\growth\left(R^nL\right) \sim \growth\left(R^nL'\right)\in  o(\growth\left(R^\infty\right))
\end{align*}

and we are done.
\end{proof}

That gives us the weaker claim in \autoref{thm:main}, now we will prove the slightly stronger version.
A strengthening of \autoref{prop:dont go left} will be as follows:

\begin{lemma}\label{lem:Rinf nonneg}
$R^\infty$ is not negligible in $\PP 2$
\end{lemma}
\begin{proof}
More precisely, we will be showing that $R^\infty$ is not negligible in $\bigcup_{n=0}^\infty R^nL$.
This is sufficient since $R^nL \supsetneq \PP 2$.
As we did in the proof of \autoref{prop:dont go left}, we will discard all words which do not contain a $\bb$, since, by the same reasoning, these words are negligible in the total.

Our strategy will be to provide an injection $f :\left(\bigcup_{n=0}^\infty R^nL'\setminus R^\infty\right) \rightarrowtail R^\infty$ which increases the length by at most a constant bound.
More formally there exists a constant $c$ such that for every $w$ in the domain, $|w| + c \geq |f(w)|$.
This is sufficient because $R^\infty$ grows exponentially, i.e.\ $\mathfrak g_n(R^\infty) \in \Theta(\lambda^k)$.
Thus
\[
\mathfrak g_{k+c}(R^\infty)\sim\lambda^c \mathfrak g_k(R^\infty) \in \Theta(\lambda^k)
\]
To define $f$ consider some word $w$ in the domain.
$w$ belongs to $R^nL$ for some $n$.
If $w$ belongs to multiple $R^nL$, let us take $n$ to be the largest value possible.
With the selected $n$, the subwords forbidden in $R^\infty$ are of the form
\[
\bb\ar^{k_0}\Bb\ar^{n+k_1}\Bb\ar^{n+k_2}\Bb\dots\Bb\ar^{n+k_{s-1}}\Bb\ar^{k_s}\bb
\]
where $s\geq 2$ and each $k_i$ is positive.
We will selected one of these subwords arbitrarily to be the ``signal segment''.
All other of these forbidden subwords we will replace with:
\[
\bb\ar^{k_0-1}\bb\Ar^{n+k_1+2}\bb\Ar^{k_2}\bb\dots\bb\Ar^{k_{s-1}}\bb\ar^{k_s-1}\bb
\]
This type of replacement cannot increase the length of the word.
For the signal segment we will replace it with the form:
\[
\bb\ar^{k_0-1}\bb\Ar\bb\Ar^{n+3}\bb\Ar^{k_1}\bb\Ar^{k_2}\bb\dots\bb\Ar^{k_{s-1}}\bb\ar^{k_s-1}\bb
\]
This type of replacement can increase the length, but by at most 4 (this happens when $s=2$).

In order to show that this is an injection we will provide a way to recover the word $w$ from the image.
We first can use the signal to recover the $n$ used in the encoding.
We find all subwords of the form $\bb\ar^i\bb\Ar\bb\Ar^{m+3}\bb$ for non-negative $i$ and $m$.
One such subword must be from the encoded signal, and none can come from other encoded forbidden words as all of those would require the first $\Ar$ syllable to have a exponent of at least $3$.
It is possible for there to be multiple subwords, but other than the signal must have been present in the initial word, so our signal is the one with the largest exponent.
Once we have recovered $n$ is is trivial to reverse the encoding on the signal segment, but we can also determine the locations of all the non-signal segments and reverse their encodings as well.
Each encoded subword must contain a word of the form $\bb\Ar^{n+k+2}\bb$ for positive $k$, and such segments are forbidden in $R^nL$, so they cannot appear elsewhere.
Thus we can reverse the encoding and $f$ is injective.

Since the only replacement which can increase the length is the signal encoding which happens exactly once, the total increase in length is at most 6.
\end{proof}

The natural extension to this question is whether $R^\infty$ is generic in $\PP 2$.
In the encoding we have provided there is a choice as to which segment is the signal.
This indicates that there may be a degree of wiggle room, suggesting $R^\infty$ could be generic.
However, the cases in where there is a choice requires multiple of the forbidden segments, which turns our to be quite rare.
In fact, counter to this initial intuition, computer simulations suggest that $R^\infty$ is not generic, with a fixed density at very roughly $1/5$.
(Uniform sampling of potentially positive words is difficult especially for larger words.
for this we sampled approximately 100 words of length about 300)

For now we offer no proof in either direction, but we will offer evidence that $R^\infty$ is the correct set to consider for this question.

\begin{theorem}\label{thm:no gen wo Rinf}
  If there is a finite strongly connected automaton which produces a generic subset of $\PP 2$, $R^\infty$ produces a generic subset of $\PP 2$.
\end{theorem}
\begin{proof}
Let $A$ be strongly connected automaton which produces a generic subset of $\PP 2$.
Assume by way of contradiction that $R^\infty$ is not generic in $\PP 2$, and so by \autoref{lem:Rinf nonneg} the density of $R^\infty$ converges to some constant value between 0 and 1.

$A$ cannot completely contain $R^\infty$, since by \cite{CSW}*{Main result} $A$ would grow faster than $R^\infty$.
Since $R^\infty$ grows like $\PP 2$ by \autoref{lem:Rinf nonneg}, this would imply that $A$ grows faster than the set in which in which it supposedly is contained.

Now let us consider the intersection $P=A \cap R^\infty$.
Since $R^\infty$ is non-negligible and $A$ is generic, $P$ is a generic proper subset of $R^\infty$.
However, the intersection of two sets produced strongly connected automata is also produced by a strongly connected automaton, so by \cite{CSW}*{Main result} the intersection cannot be generic in $R^\infty$.
\end{proof}

\section{An automaton for the potentially positive words}

In our proof of \autoref{thm:no gen wo Rinf}, we provide an argument which can show that there is no finite strongly connected automaton which produces exactly the potentially positive words.

In light of this we if we can produce $\PP 2$ using a more general type of automaton.
Most directly we might lift the requirement that words are produced by closed loops allowing specified start and end states and considering the images of all paths from a start state to an end state.
(Our closed loop requirement can be translated to such an automaton by creating one copy of the automaton for each node and a node in each copy as the start and end state.)
Such an automaton is called a ``finite state machine'' (FSM) and a language produced by an FSM is called ``regular''.

\begin{question}\label{q:isreg}
  Is $\PP{r}$ a regular language for any $r\geq 2$?
\end{question}

For those whose background is mainly in dynamics, this is equivalent to asking if the shift of potentially positive words is {\it sofic}.

If the answer to these questions were ``Yes'', it would mean we could implement an algorithm to identify cyclically-reduced potentially positive words of length $n$ with worst-case time complexity $O(n)$ and constant space.
Indeed this is an enticing prospect.

The reader may have already guessed that, since this is posed as a question rather than a proposition, this dream will not come to pass.
We will be showing that no automaton can exist.
In fact, we will be showing a more general result.

Finite state machines are equivalent to Turing machines without a memory tape.
Hence they are called {\it finite} state machines, they are a a computation model with finite, i.e.\ bounded by a constant, memory.
We already know from \cite{KHOO} that a deterministic Turing machine can recognize $\PP 2$.
Thus unrestricted memory is enough.
In fact the algorithm provided uses only linear space ($\mathrm{DSPACE}(n)$), and so $\PP 2$ can be identified by a Turing machine with a linear bound on its memory.
Such an automaton is called a ``linear bounded automaton'', and a language which can be recognized by one is called ``context-sensitive''. 
Other classes of automata exist between linear bounded automata and finite sate machines.
For example context-free languages are those which can be identified with a push down automaton.

\begin{theorem}\label{thm:noCFL}
For $r\geq 2$, $\mathrm{PP}_r$ is not a well-nested multiple context-free language.
\end{theorem}

Well-nested multiple context-free languages (well-nested MCFLs) are a class of languages containing all regular languages.
Well-nested MCFLs are not a widely known class of languages, and so it is very likely that readers of this paper will be unfamiliar with them.
We will not make any attempt to change that.
A definition can be found in \cite{MCFLpump}, but our proof will not use the technical details of the definition.
Since the regular languages are contained within the well-nested MCFLs, \autoref{thm:noCFL} will imply the answer to \autoref{q:isreg} is ``No''.
In fact, because we avoid the technical details and instead use existing results, the reader can substitute ``well-nested MCFL'' for ``regular'' and still completely understand the proof.
In fact other well known families of languages such as context-free languages, ETOL languages, and tree-adjoining languages are also contained within the well-nested MCFLs, and can be substituted in the same way.

We will establish the two instrumental lemmata about well-nested MCFLs:

\begin{lemma}[Pumping lemma for well-nested MCFLs \cite{MCFLpump}]
  \label{lem:pump}
  If $L \subseteq \Sigma^\ast$ is a well-nested MCFL, there exists some $k \geq 1$ (called the \defnsb{pumping dimension} of $L$) and some $p\geq 1$ (called the \defnsb{pumping length} of $L$) such that, every string $s\in L$ with $|s|_\Sigma\geq p$ can be written as
  \[
  s=v_0\,w_1\,v_1\,w_2\,v_2\,\dots\,w_k\,v_k
  \]
  where:
  \begin{itemize}
      \item $\left\{w_i\mid 1\leq i\leq k\right\}\neq \{\varepsilon\}$
      \item $\forall n\in \mathbb{Z}_{\geq 0}:v_0\,w_1^n\,v_1\,w_2^n\,v_2\,\dots\,w_k^n\,v_k \in L$
  \end{itemize}
\end{lemma}

Naturally this lemma applies to all weaker classes of language.
However, a much simpler pumping lemma exists for regular languages.
See \cite{HopcroftUllman}*{Thm.\ 4.1}.\footnote{For context-free languages see \cite{HopcroftUllman}*{Thm.\ 7.18}. For tree adjoining languages see \cite{TAG}*{Thm.\ 4.7}.}
Both can be used to exclude $\mathrm{PP}_2$ from the respective classes using the same techniques we will employ.

\begin{lemma}[\cite{SeikiKato}*{Prp.\ 6}\footnote{The presentation in \cite{SeikiKato} significantly differs from the presentation here. One can refer to \cite{MCFLpump}*{pg.\ 323} which reports the result in a more harmonious way.}]
\label{lem:intersection}
For each well-nested MCFL $L_1 \subseteq \Sigma^\ast$ there is a $k$ such that for every regular language $L_2\subseteq \Sigma^\ast$, $L_1\cap L_2$ is a well-nested MCFL with pumping dimension $k$.
\end{lemma}

For the equivalent on regular languages see \cite{Jewels}*{Thm.\ 2.7}.\footnote{For context-free languages see \cite{HopcroftUllman}*{Thm.\ 7.27}, and for tree adjoining languages see \cite{TAG}*{Thm.\ 4.5}.}

From here \autoref{thm:noCFL} will be a rather simple corollary of the following:

\begin{lemma}
  $\mathrm{PP}_2$ is not a well-nested MCFL.
\end{lemma}
\begin{proof}
  Assume by way of contradiction that $\mathrm{PP}_2$ is a well-nested MCFL with pumping dimension $k$.
  Let us consider the language $\Lambda_k$ consisting of words of the form:
  \begin{align*}
    \Bb\ar^{n_1}\Bb\ar^{n_2}\Bb\ar^{n_2}\Bb\dots\Bb\ar^{n_k}\Bb\ar\bb\Ar^m\bb\ar
  \end{align*}
  where $n_1,\dots,n_k,m\geq 1$.
  Since $\Lambda_k$ is regular, it follows from \autoref{lem:intersection} that $\mathrm{PP}_2$ is a well-nested MCFL only if $\Lambda_k\cap\mathrm{PP}_2$ is a well-nested MCFL.
  We will show first that if a word $w\in \Lambda_k$ is potentially positive then $\forall i:n_i\geq m-1$.

  Let us consider a word $s$ which is in the language $\Lambda_k$ but has some $n_i < m-1$.
  We note that $s$ satisfies \goldC\ with $x=\Bb$ and $y=\ar$.
  We will show that this word no longer satisfies \goldC\ after applying the automorphism $\bb\mapsto\bb\ar^n$ where $n$ is the minimum $n_i$ in $s$.
  After applying the automorphism, the syllable $\ar^{n}$ becomes empty, creating a $\Bb\Bb$ subword, so $x \neq \Bb$.
  By \autoref{lemma:switch}, the only other option is $x=\Ar$.
  However, $\Ar\Ar$ must appear as a subword because $\bb\Ar^m\bb$ appeared in the original word and we have only removed $n\leq m-2$ from the exponent.
  Since $s$ is automorphic to a word which does not satisfy \goldC\, it is not potentially positive.
  
  Now let the pumping length of $L_k\cap \mathrm{PP}_2$ be $p$.
  We consider the word
  \begin{align*}
    s=\Bb\ar^p\Bb\ar^p\Bb\dots\Bb\ar^p\Bb\ar\bb\Ar^{p+1}\bb\ar \in L_k.
  \end{align*}
  This word is potentially positive as we can demonstrate:
  \begin{align*}
    s&=\Bb\ar^p\Bb\ar^p\Bb\dots\Bb\ar^p\Bb\ar\bb\Ar^{p+1}\bb\ar \\
    &\mapsto\Bb^{k+1}\ar\bb\Ar\bb\ar \tag{$\bb\mapsto \bb\ar^p$} \\
    &\mapsto\ar\bb\Ar\bb\ar \tag{$\ar\mapsto \bb^{k+1}\ar$} \\
    &\mapsto\ar\bb\bb\ar\ar \tag{$\bb\mapsto \bb\ar$}
  \end{align*}
  And thus $s \in L_k\cap \mathrm{PP}_2$.
  
  $|s| = (k+1)(p+1)+5\geq p$ thus \autoref{lem:pump} should apply to this word.
  No $w_i$ can contain any of the underlined sections in the following since removing them (i.e.\ $v_0\,w_1^0\,v_1\,w_2^0\,v_2\,\dots\,w_k^0\,v_k$) makes the word no longer $L_k$:
  \begin{equation*}
    \underline\Bb\ar^p\underline\Bb\ar^p\underline\Bb\dots\underline\Bb\ar^p\underline{\Bb\ar\bb}\Ar^{p+1}\underline{\bb\ar}
  \end{equation*}
  There are $k+1$ separate segments remaining, so it is impossible for all of them to contain a $w_i$.
  Each segment is either a $\ar$ or $\Ar$ syllable.
  No $w_i$ can contain any part of the remaining $\Ar$ syllable, since increasing the power on $\Ar$ without increasing the powers of all the other $\ar$ syllables would make the word no longer potentially positive by the argument above.

  Furthermore, if some non-empty $w_i$ contained some part of a remaining $\ar$ syllable, removing the $w_i$ subword would decrease the exponent on that syllable.
  That would make it less than $m-1=p$.
  $m$ cannot decrease because the $\Ar$ syllable contains no $w_j$.
  Thus the word would no longer be potentially positive.

  There must be a non-empty $w_i$, but none can be contained in any part of $s$, so we have a contradiction and $\mathrm{PP}_2$ cannot be a well-nested MCFL.
\end{proof}

Now \autoref{thm:noCFL} will follow easily:
\begin{proof}[Proof of \autoref{thm:noCFL}]
By the main result of \cite{GC}, $\mathrm{PP}_r\cap F_n = \mathrm{PP}_n$ when $r\geq n$.
Since $F_r$ is regular, we can apply \autoref{lem:intersection} to see that $\mathrm{PP}_r$ for $r\geq 2$ is a well-nested MCFL only if $\mathrm{PP}_2$ is a well-nested MCFL.
Since $\mathrm{PP}_2$ is not a well-nested MCFL, no higher rank is either.
\end{proof}

\begin{corollary}
  There is no algorithm to decide cyclically-reduced potentially positive words using constant space.
\end{corollary}
This follows from a classic result in complexity theory (see \cite{HoMLC}) which tells us that if a set is decidable in constant space it is regular.

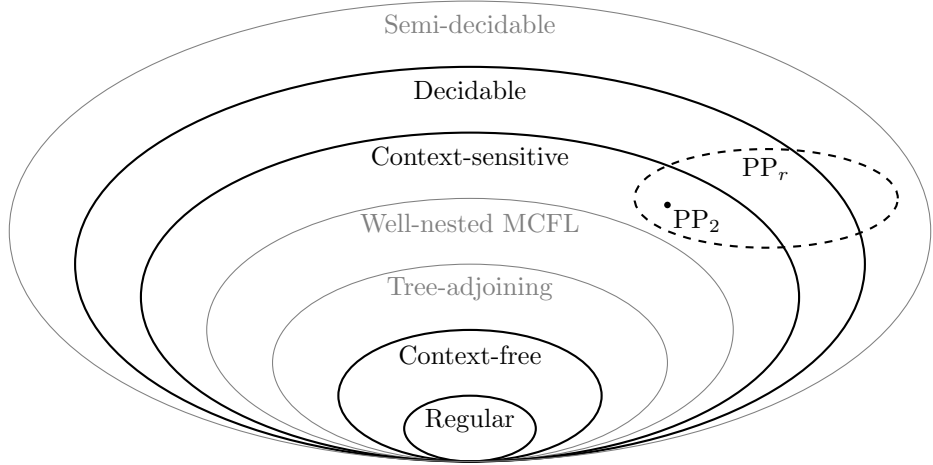
\begin{figure}[h]
\caption{
  An Euler diagram indicating the placement of potential positivity in relation to various classes of formal languages.
  The classes in the classical Chomsky hierarchy are drawn thicker and darker than the other classes.
  The dotted line indicates the greatest possible span of $\PP{r}$ for $r\geq 2$ with our current understanding.
}
\begin{tikzpicture}[scale=0.875]
    \draw[thick] (0,0) ellipse (1 and 0.5) node at (0,.125) {Regular};
    \draw[thick] (0,.5) ellipse (2 and 1) node at (0,1.125) {Context-free};
    \draw[thin,gray] (0,1) ellipse (3 and 1.5) node at (0,2.125) {Tree-adjoining};
    \draw[thin,gray] (0,1.5) ellipse (4 and 2) node at (0,3.125) {Well-nested MCFL};
    \draw[thick] (0,2) ellipse (5 and 2.5) node at (0,4.15) {Context-sensitive};
    \draw[thick] (0,2.5) ellipse (6 and 3) node at (0,5.15) {Decidable};
    \draw[thin,gray] (0,3) ellipse (7 and 3.5) node at (0,6.15) {Semi-decidable};
    \draw[thick,dashed] (4.5,3.5) ellipse (2 and 0.75) node at (4.5,3.95) {$\PP{r}$};
    \fill (3, 3.4) circle (.05) node at (3.45,3.15) {$\PP 2$};
\end{tikzpicture}
\end{figure}

\section{Towards non-trivial upper bounds}
Finding an upper bound for $F_3$ has proven to be more difficult.
Whereas in $F_2$ we were able to use the algorithms in \cite{Goldstein} and \cite{KHOO} to eliminate words and produce an upper bound, no algorithm is currently known for higher ranks.
In fact, the words known to not be potentially positive in $F_3$ are fairly restricted.
\begin{enumerate}
\item\label{itm:commuators} Words in the commutator subgroup $\left[F_3,F_3\right]$.
\item\label{itm:rank reduce} Words in $F_2$ which are not potentially positive under the automorphisms of $F_2$ (see \cite{GC}) and words which are automorphic to them.
\end{enumerate}
Theorem 3.9 in \cite{KHOOS} can be used to show that (\ref{itm:rank reduce}) is exponentially negligible.
A slight narrowing of that theorem to fit this paper:
\begin{theorem}[Thm 3.9]
  If $w\in F_r$ is a cyclically reduced word in the orbit of a word $v\in F_{r-1}$ then $w$ does not contain a primitivity blocking subword.
\end{theorem}
By \cite{CSW}*{Example 1} if there is a forbidden subword for some subset of a free group then that set is negligible.

So then, (\ref{itm:commuators}) seems most promising.
If we recall the Cohen-Grigorchuk theorem:
\begin{theorem}[Cohen-Grigorchuk]\label{thm:Cohen-Grigorchuk}
  Let $\varphi : F_r \rightarrow G$ be a surjective group homomorphism.
  $G$ is amenable if and only if
  \[
  \lim_{n\rightarrow\infty}\sqrt[n]{\growth_n(\ker(\varphi))}= 2r-1
  \]
\end{theorem}
We refer the reader to \cite{Szwarc} for a proof.
The implication here is that, because the abelianization of $F_r$ is amenable, the commutator subgroup coarsely (up to a sub-exponential factor) grows like $(2r-1)^n$.
This gives us the beginning of a non-trivial upper bound:

\begin{corollary}
  The set of potentially positive words in $F_r$ is not exponentially generic in the set of all words.
\end{corollary}
\begin{proof}
  The set of all words grow as $\Theta((2r-1)^n)$.
  If the set of potentially positive words were to be exponentially generic, then its compliment, which contains the commutator subgroup, would have to grow as $O(\varepsilon^n(2r-1)^n)$ for some $\varepsilon< 1$.

  Thus
  \[
  \lim_{n\rightarrow\infty}\sqrt[n]{\growth_n(\left[F_r,F_r\right])} \leq \varepsilon(2r-1) < 2r-1
  \]
  a contradiction with \autoref{thm:Cohen-Grigorchuk}.
\end{proof}

A more precise asymptotic for the growth of $\left[F_r,F_r\right]$ can be obtained from \cite{Rivin2} as
\begin{equation}
\dfrac{1}{\sqrt{n}^r}\left(2r-1\right)^n
\end{equation}
This tells us that the commutator subgroup is in fact negligible (only not exponentially so).

We summarize the current situation in this table:
\begin{center}
\begin{tabular}{|r|c|c|c|c|}
\hline
& pos. & prim. & p. pos.\footnote{The numeric values provided in this column are approximate to the thousandths place. We provide the method by which the exact values can be derived elsewhere.} & all \\
\hline\hline
$F_2$ & $\Theta\left(2^n\right)$ & $\Theta\left(\sqrt{3}^n\right)$ & $\Theta\left(2.505^n\right)$ & $\Theta(3^n)$ \\ \hline
$F_3$ & $\Theta\left(3^n\right)$ & $\Theta\left(n3^n\right)$ & $\Omega\left(4.024^n\right)$ & $\Theta(5^n)$ \\ \hline
$F_4$ & $\Theta\left(4^n\right)$ & $\Theta\left(n5^n\right)$ & $\Omega\left(5.746^n\right)$ & $\Theta(7^n)$ \\ \hline
$F_5$ & $\Theta\left(5^n\right)$ & $\Theta\left(n7^n\right)$ & $\Omega\left(7.509^n\right)$ & $\Theta(9^n)$ \\ \hline
$F_6$ & $\Theta\left(6^n\right)$ & $\Theta\left(n9^n\right)$ & $\Omega\left(9.290^n\right)$ & $\Theta(11^n)$ \\ \hline
$F_7$ & $\Theta\left(7^n\right)$ & $\Theta\left(n11^n\right)$ & $\Omega\left(11.083^n\right)$ & $\Theta(13^n)$ \\ \hline
$F_r$ & $\Theta\left(r^n\right)$ & $\Theta\left(n(2r-3)^n\right)$ & $\Omega\left((2r-3)^n\right)$ & $\Theta\left((2r-1)^n\right)$ \\ \hline
\end{tabular}
\end{center}

\bibliography{growth}

\end{document}